\documentclass[11pt,amscd,amssymb,verbatim]{amsart}
\usepackage{amssymb,amsmath,amsthm,mathrsfs,graphicx,xcolor,tikz}
\usepackage{amsmath,amssymb,amsfonts,amsthm,graphicx}
\usepackage[mathscr]{euscript}
\usepackage{epsfig}
\usepackage{subfigure}
\usepackage{graphicx}
\usepackage{amsbsy}
\usepackage{cite}
\usepackage{amsmath}
\usepackage{natbib}
\makeatletter
\renewenvironment{proof}[1][\proofname]{\par
\pushQED{\qed}%
\normalfont \topsep6\p@\@plus6\p@\relax
\trivlist
\item\relax
{\bfseries
#1\@addpunct{.}}\\
}{%
\popQED\endtrivlist\@endpefalse
}
\makeatother
\usepackage{color}
\usepackage{natbib}
\usepackage{nameref}
\theoremstyle{definition}
 \usepackage[colorlinks,urlcolor=blue]{hyperref}
    \usepackage{url}
\newtheorem{exmp}{Example}[section]
\usepackage{xspace}

\numberwithin{equation}{section}
\newtheorem{theorem}{Theorem}

\newtheorem{corollary}{Corollary}
\newtheorem{remark}{Remark}
\newtheorem{definition}{Definition}
\newtheorem{proposition}{Proposition}
\begin{document}

\title[Some results about varextropy] {Some results about varextropy  and weighted varextropy functions \vspace{1cm}}
 \maketitle
 \begin{center}
\author{ F. Goodarzi\footnote{Corresponding author: {\tt f-goodarzi@kashanu.ac.ir}} $^1$,  R. Zamini$^2$\\
\small $^1$   Department of Statistics, Faculty of Mathematical Sciences, University of Kashan, Kashan,\ Iran.
	\small $^2$  Department  of Mathematics, Faculty of Mathematical Sciences and  Computer, Kharazmi University, Thehran, Iran. \\}
\end{center}
\date{}

\vspace{0.09 in}

\begin{abstract}
 In this paper, we investigate several properties of the weighted varextropy measure and 
obtain it for specific  distribution functions, such as the equilibrium and weighted ditributions. We also obtain bounds for the weighted varextropy, as well as for weighted residual varextropy and weighted past varextropy. 
 Additionally, we derive an expression for the varextropy of the lifetime of coherent systems. 
A new stochastic ordering, referred to as weighted varextropy ordering, is introduced, and some of its key properties are explored.
Furthermore, we propose two nonparametric estimators for the weighted varextropy function. 
A simulation study is conducted to evaluate the performance of these estimators in terms of mean squared error (MSE) and bias. Finally, we provide a characterization of the reciprocal distribution based on the weighted varextropy measure. Some tests for reciprocal distribution are constructed by using the proposed estimators and the powers of the tests are compared with the powers of kolmogorov-Smirnov (KS) test. Application to real data is also reported.
\end{abstract}

\vspace{0.09 in}

\section{Introduction}
Let $X$ be an absolutely continuous random variable with probability density function (pdf) $f$, distribution function $F$ and survival function $\overline F$. Entropy, as an uncertainty measure is defined as  the expectation of the information content of $X$, and is given by \cite{Shannon} as follows
\begin{align}
H(X)=-\int_{-\infty}^{+\infty}f(x)\log f(x)dx.
\end{align}

It is remarked that, the variance entropy (varentropy) of a random variable $X$ is defined as 
\begin{align}\label{VEn}
VE(X)&=Var[-\log f(X)]=E[(-\log f(X))^2]-[H(X)]^2\nonumber\\&=
\int_{-\infty}^{+\infty}f(x)[\log f(x)]^2dx-\left[\int_{-\infty}^{+\infty}f(x)\log f(x)dx\right]^2.
\end{align} 
The varentropy is positive functional and does not depend on location and scale parameters, i.e. $Var[-\log f(X)]= Var[-\log g(X)]$, where $f(x)=\sigma^{-1}g((x-\mu)/\sigma)$.

Another measure of uncertainty is the extropy, which was defined by \cite{Lad} as dual to the entropy and is given as:
\begin{align}\label{ex}
J(X)=-\frac{1}{2}\int_{-\infty}^{+\infty}f^2(x)dx.
\end{align}
If the extropy of random variable $X$ is less than that of random variable $Y$, then $X$ is said to have more uncertainty than $Y$. \cite{Lad} displayed some interesting properties of this measure such as the maximum extropy distribution and some statistical applications.  For more studies  on extropy, see \cite{Qiu2017}, \cite{Qiu2018a}, \cite{Qiu2018b}, \cite{Yang},  among others. 

The weighted extropy can be seen as a measure to quantify the amount of uncertainty present in a random variable. This uncertainty depends on the probabilities assigned to different events and their relevance to the specific qualitative characteristic being considered.
The weighted extropy, as proposed by \cite{Sattar} and \cite{Bala}, is defined as
\begin{align}\label{exw}
J^w(x)=-\frac{1}{2}\int_{-\infty}^{+\infty}xf^2(x)dx.
\end{align}
In general, the weighted extropy measure can be defined as 
\begin{align}
J_\phi^w(x)=-\frac{1}{2}\int_{-\infty}^{+\infty}\phi(x)f^2(x)dx,
\end{align}
where $\phi(X)$ is considered as the weight or utility function and it is obvious that $\phi(X) = 1$ and $\phi(X) = X$, respectively results to 
Equations \eqref{ex} and \eqref{exw}. 

Recently \cite{Vaselabadi} introduced a measure of uncertainty which can be used as an alternative measure to \eqref{VEn}. This measure is known as varextropy and defined as 
\begin{align}\label{eqve}
VJ(X)&=Var\left[-\frac{1}{2}f(X)\right]\nonumber\\&=\frac{1}{4}E(f^2(X))-J^2(X).
\end{align}   
As they stated, when the extropy of two variables are equal, the varextropy is useful to determine which extropy would be more appropriate to measure uncertainty.  They also stated that the varextropy measure is more flexible than varentropy, i.e. the varextropy is free of the model parameters in some distributions. 
It is well-known that for exponential distribution, varextropy is independent of the lifetime of the system and remains unchanged for the system and remains unchanged for the symmetric distributions. \cite{Vaselabadi} showed, for the continuous symmetric distribution $F(x)$, varextropy of the series and parallel systems are the same. They obtained some of properties this measure based on order statistics, record values and proportional hazard rate models. \cite{Good} also showed that, if extropy and varextropy for the series and parallel systems are equal then $F(x)$ is a symmetric cdf. In addition, she obtained a lower bound for varextropy. 

Sometimes, in addition to the same extropy for two random variables, varextropy may also be the same for these two variables, and therefore another criterion should be used to express their uncertainty. To this end, \cite{ChGu} defined weighted extropy as follows:
\begin{align}\label{vew}
VJ^w(X)=\frac{1}{4}\int_{-\infty}^{+\infty}x^2f^3(x)dx-(J^w(X))^2.
\end{align}
Recently, \cite{Zhang2025} defined the following general form for weighted varextropy as
\begin{align}\label{4Tir}
VJ_{\phi}^w(X)=\frac{1}{4}\int_{-\infty}^{+\infty}{\phi}^2(x)f^3(x)dx-(J_{\phi}^w(X))^2.
\end{align}
They also introduced the concept of the weighted residual varextropy and investigated its behavior under arbitrary monotonic transformations. Furthermore, they introduced non-parametric estimators for weighted varextropy and weighted residual varextropy.

Before presenting the main results of the paper, it is necessary to provide a theorem and some the preliminary definitions.
\begin{theorem}
The classical Hardy inequality reads 
\begin{eqnarray}\label{T1.1}
\int_{0}^{+\infty}\Big(\frac{1}{x}\int_{0}^{x}f(t)dt\Big)^pdx\leq\Big (\frac{p}{p-1}\Big)^p\int_{0}^{+\infty}f^p(x)dx, \ \ p>1,
\end{eqnarray}
where $f$ is a non-negative function such that $f\in L^p(0, \infty)$.  Hardy proved  \eqref{T1.1} in \cite{Hardy}.
\end{theorem}
\begin{definition}
\begin{enumerate}
Suppose that $X$ and $Y$ are two random variables with density functions $f$ and $g$ and distribution functions $F(x)$ and $G(x)$, respectively. Then, 
\item[(1)] The random variable $X$ is said to be smaller than $Y$ in the varextropy order, denoted
by $X\leq_{VJ}Y$, if $VJ(X)\leq VJ(Y)$.

\item[(2)] $X$ is smaller than $Y$ in the dispersive order, denoted by $X\leq_{disp}Y$, if $f (F^{-1}(v))\geq g(G^{-1}(v))$
for all $v\in (0, 1)$, where $F^{-1}$ and $G^{-1}$ are right continuous inverses of $F$ and $G$, respectively.
\end{enumerate}
\end{definition}
\section{Weighted varextropy measure}
In recent years, the analysis of variability in uncertainty measures has attracted considerable attention within the field of information theory. 
In this context, the concepts of varextropy and weighted varextropy have been investigated by several researchers. 
In this paper, we aim to conduct a more comprehensive study of the weighted varextropy measure. First, we present some examples to illustrate the importance of this measure.
Note that, in equation \eqref{4Tir} when $\phi(x)=1$ and $\phi(x)=x$, equations \eqref{eqve} and \eqref{vew} are obtained, respectively.
\begin{exmp}
Suppose that $X$ and $Y$ have the beta distributions with parameters $(a, 1)$ and $(1,a)$, respectively for $a>0$. Then $J(X)=J(Y)=-\frac{a^2}{2(2a-1)}$ and
$VJ(X)=VJ(Y)=\frac{a^3(a-1)^2}{4(3a-2)(2a-1)^2}$, while weighted varextropy of $X$ and $Y$ are different and equal to $\frac{a^2}{48}$ and 
$\frac{a^2(5a^2-5a+2)}{48(9a^2-9a+2)(2a-1)^2}$, respectively.
\end{exmp}
\begin{exmp}
Let $X$ and $Y$ have the beta distributions with the following pdf 
\begin{eqnarray*}
\label{fifteen}
f_{X}(x)=\left\{\begin{array}{ll}a(a+1)x^{a-1}(1-x)&a>0,\, 0<x<1,\\
0&otherwise.
\end{array}\right.
\hspace{0.5cm}f_{Y}(y)=\left\{\begin{array}{ll}a(a+1)y(1-y)^{a-1}&a>0,\, 0<y<1,\\
0&otherwise.
\end{array}\right.
\end{eqnarray*}
\end{exmp}
Then extropy and varextropy of $X$ and $Y$ are equal and are obtained  $-\frac{a(a^2+2a+1)}{2(4a^2-1)}$ and $\frac{a^3(5a^6+6a^5-7a^4-6a^3+9a^2+8a+1)}{4(27a^3-18a^2-3a+2)(4a^2-1)^2}$, respectively. Also we get \[VJ^w(X)=\frac{a^2(5a^4+15a^3+17a^2+9a+2)}{48(9a^2+9a+2){(2a+1)}^2},\] \[VJ^w(Y)=\frac{a^2(373a^6+746a^5+308a^4-130a^3-13a^2+104a+52)}{48(81a^4-45a^2+4){(4a^2-1)}^2},\]
and $VJ^w(Y)$ is fininit for $a\ne \frac{1}{3}$, $a\ne \frac{1}{2}$ and $a\ne \frac{2}{3}$.
In Figure \ref{pic1}, the graph of $VJ^w(X)$ and $VJ^w(Y)$ in terms of $a$ is depicted for $\frac{2}{3} <a< 9$, which acknowledges that $VJ^w(X)>VJ^w(Y)$
for $a<2$ and $VJ^w(X)<VJ^w(Y)$ for $a>2$.  
\begin{figure}
\begin{center}
\includegraphics[width=5cm, height=5cm]{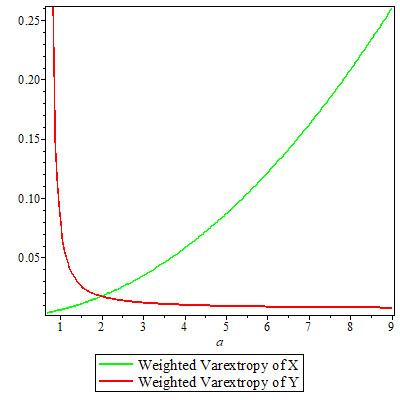}
\caption{\footnotesize{Plots of $VJ^w(X)$ and $VJ^w(Y)$.}}\label{pic1}
\end{center}
\end{figure}
\begin{exmp}
Suppose that $X$ has a Weibull distribution with distribution function $F(x)=1-e^{-\lambda x^{\alpha}}, \, x>0$ for $\alpha>0$ and $\lambda>0$. Then the weighted varextropy is equal to $\frac{\alpha^2}{54}-\frac{1}{64}$, that does not depend on $\lambda$ parameter. Also, if $\alpha=1$, that is  $X$ has exponential distribution then $VJ^w(X)=\frac{5}{1728}$.  
\end{exmp}
\begin{exmp}\label{eglap}
Let $X$ has Laplace distribution with pmf $f(x)=\frac{1}{2\beta}\exp\{-\frac{|x|}{\beta}\}, \, \beta>0$, then $VJ^w(X)=\frac{1}{216}$, whereas $VJ(X)=\frac{1}{192\beta^2}$. Therefore, the advantage of weighted varextropy is that, unlike varextropy, it does not depend on the parameter.
\end{exmp}
In the following, we discuss some properties of weighted varextropy, including its behavior for weighted distributions and its changes under monotonic transformations.
\begin{proposition}
Let $X$ be a non-negative absolutely  continuous random variable with $\mu=E(X)$. Then considering the equilibrium function $f_Y(x)=\frac{\overline F(x)}{\mu}, \, 0<x<+\infty$ where $\overline F(x)=1-F(x)$ is survival function of $X$, we have
\begin{align}
VJ^w(Y)=\frac{1}{4}\left[\int_{0}^{1}\frac{(1-u)^3}{\mu^3}\frac{(F^{-1}(u))^2}{f(F^{-1}(u))}du-\left(\int_{0}^1\frac{(1-u)^2}{\mu^2}\frac{F^{-1}(u)}{f(F^{-1}(u))}du\right)^2\right].
\end{align}
\end{proposition}
\begin{proposition}
Let $Y$ is an absolutely continuous random variable with weighted distribution $f_Y(x)=\frac{\delta(x)f(x)}{E[\delta(X)]}, x\in \mathbb{R}$ with $0<E[\delta(X)]<+\infty$. Then 
\begin{align}
VJ^w(Y)=\frac{1}{4E^2[\delta(Y)]}Var[Y\delta(Y)f(Y)].
\end{align}
\end{proposition}
\begin{proposition}
Let $Y=g(X)$ is a strictly increasing function of $X$, then
\begin{align}
VJ^w(Y)=\frac{1}{4}Var\left[\frac{g(X)}{g^{\prime}(X)}f(X)\right].
\end{align}
Note that if $Y=aX+b$, with assumption $a\ne 0$, then $VJ^w(Y)=VJ^w(X)+\frac{b^2}{a^2}VJ(X).$ Clearly, if $b=0$ then $VJ^w(Y)=VJ^w(X)$.
Also, if $Y=F_X(X)$ then we obtain $VJ^w(Y)=\frac{1}{48}$.
\end{proposition}
\begin{exmp}\label{egnor}
Let $Z$ have a standard normal distribution, then $VJ^w(Z)=\frac{1}{72}\frac{\sqrt 3}{\pi}$ and $VJ(Z)=\frac{1}{24}\frac{\sqrt 3}{\pi}-\frac{1}{16\pi}$, and thus weighted varextropy of $X=\sigma Z+\mu$ where $-\infty<\mu<\infty, \sigma>0$ is $VJ^w(X)=\frac{1}{72}\frac{\sqrt 3}{\pi}+\frac{\mu^2}{\sigma^2}(\frac{1}{24}\frac{\sqrt 3}{\pi}-\frac{1}{16\pi})$. Furthermore, if $X\sim N(0, \sigma^2)$ then $VJ^w(X)=\frac{1}{72}\frac{\sqrt 3}{\pi}$.
\end{exmp}
\begin{proposition}
If $X$ is a symmetric random variable with respect to a finite mean $\mu=E[X]$,  i.e.,  $F(x+\mu)=1-F(\mu-x)$, then  $VJ^w(X-\mu)=VJ^w(\mu-X)$.
\end{proposition}
\begin{exmp}
Consider  a random variable $X$ with piecewise constant probability density function  
\[f(x)=\sum_{j=1}^na_j1_{[j-1, j)}(x),\]
where $a_j\geq 0, j=1, \ldots, n, \sum_{j=1}^na_j=1$, and $1_{[j-1, j)}(x)$ is the indicator function.
Then, the varextropy and the weighted varextropy of $X$ are  
\begin{align}
VJ(X)&=\frac{1}{4}\left[\int_{0}^n\sum_{j=1}^na_j^31_{[j-1, j)}(x)dx\right]-(J(X))^2\nonumber
\\&=\frac{1}{4}\left[
\sum_{j=1}^n\int_{k-1}^ka_j^3dx
\right]-(J(X))^2=\frac{1}{4}\left[\sum_{j=1}^na_j^3-(\sum_{j=1}^na_j^2)^2\right],
\end{align}
\begin{align}
VJ^w(X)&=\frac{1}{4}\left[\int_0^n x^2\sum_{j=1}^na_j^31_{[j-1, j)}(x)dx\right]-(J^w(X))^2\nonumber\\&=
\frac{1}{4}\left[\sum_{j=1}^n\int_{j-1}^j x^2a_j^3dx\right]-(J^w(X))^2\nonumber\\&=\frac{1}{12}\sum_{j=1}^n(3j^2-3j+1)a_j^3-\left(\frac{1}{4}\sum_{j=1}^n(2j-1)a_j^2\right)^2.
\end{align}
If for every $j=1, \ldots, n$, $a_j=\frac{1}{n}$, then the values of varextropy and weighted varextropy will be $0$ and $\frac{1}{48}$, respectively.
\end{exmp}
\begin{definition}
If $X$ and $Y$ are absolutely continuous random variables, we define the bivariate version of weighted varextropy, as follows
\begin{align}
VJ^w(X, Y)=\frac{1}{16}\left[\int_{-\infty}^{+\infty}\int_{-\infty}^{+\infty}x^2y^2f^3_{X, Y}(x, y)dxdy-{\left(\int_{-\infty}^{+\infty}\int_{-\infty}^{+\infty}xyf^2_{X, Y}(x,y)dxdy\right)}^2\right], 
\end{align}
If $X$ and $Y$ are independent then we will have
\begin{align}
VJ^w(X, Y)=VJ^w(X)\left[VJ^w(Y)+(J^w(Y))^2\right]+\left[(J^w(X))^2\right]VJ^w(Y).
\end{align}
\end{definition}
\begin{exmp}
Let $X$ and $Y$ have the bivariate exponential distribution with probability density function 
\begin{align}
f_{X, Y}(x,y)=((1+\theta x)(1+\theta y)-\theta)e^{-(x+y+\theta xy)}, \, x>0,\, y>0, \, 0\leq \theta \leq 1,
\end{align}
then we compute the bivariate weighted varextropy as 
\begin{align}
VJ^w(X, Y)=\frac{1}{2916\theta^2}\left(3e^{\frac{3}{\theta}}E_1\left(\frac{3}{\theta}\right)(\theta^2-3\theta-9)+\theta^2+9\theta\right)
-\frac{1}{4096\theta^2}
\left(2e^{\frac{2}{\theta}}E_1\left(\frac{2}{\theta}\right)(\theta+2)-\theta\right)^{2},
\end{align}
where $E_1(\frac{i}{\theta})=\int_{i}^{+\infty}\frac{e^{-\frac{x}{\theta}}}{x}dx$ for $i=2, 3$.
It is necessary to mention, if $\theta=0$ or $X$ and $Y$ are independent and have standard exponential distribution then 
$VJ^w(X, Y)=\frac{295}{2985984}.$
\end{exmp}
We now aim to find a lower bound for the weighted varentropy. To this end, we refer to Theorem 4.1 in \cite{Afen}, where lower bounds for the variance of  a function $g(X)$, are derived for specific distributions including the normal, uniform, and gamma distributions.
In Example \eqref{egnor}, we computed the weighted varextropy for the normal distribution. We now aim to obtain a lower bound for it as well. Additionally, we will obtain a lower bound for the weighted varextropy of the inverse gamma distribution. For this purpose , we compute the derivatives of the function $g(x)=-\frac{1}{2} f(x)$ up to the third order. First, for the normal distribution, since 
\begin{align}
Var[g(X)]\geq \sum_{k=1}^n\frac{(\sigma^2)^k}{k!}E^2\left[g^{(k)}(X)\right],
\end{align}
where $g^{(k)}$ is the $k$th derivative of $g$, thus for $n=3$, we have 
\begin{align}
VJ^w(X)\geq \sum_{k=1}^3\frac{(\sigma^2)^k}{k!}E^2\left[-\frac{1}{2}f^{(k)}(X)\right]
=\frac{11}{512\pi}+\frac{1}{128}\frac{\mu^2}{\sigma^2\pi}.
\end{align}
Note that, for example, if we consider $\mu=3$ and $\sigma^2=2$, the exact value and the lower bound for the weighted varextropy are 
$0.0146$ and $0.0124$, respectively, indicating the high accuracy of this lower bound.

Now if $X$ have inverse gamma distribution with probability distribution function $f(x)=\frac{\beta^{\alpha}}{\gamma(\alpha)}$, then using Theorem 4.3 in \cite{Afen}, we obtain a lower bound for the  variance $g(X)$ as follows
\begin{align}
Var[g(X)]\geq \sum_{k=1}^{n}\frac{\Gamma(\alpha)E^2[X^{2k}g^{(k)}(X)]}{\Gamma(\alpha-2k)\beta^{2k}k!{(\alpha-1)}^k\prod_{j=k-1}^{2k-2}(1-\frac{j}{\alpha-1})}.
\end{align}
Hence for $n=3$, 
\begin{align}
VJ^w(X)\geq \frac{1}{512}\frac{(4\alpha^5+15\alpha^4-961\alpha^3+7575\alpha^2-22317\alpha+21636)(\Gamma(\alpha-\frac{3}{2}))^2}{\pi(\alpha-1)(\Gamma(\alpha))^2}.
\end{align}
Recently the weighted residual varextropy is defined by \cite{Zhang2025}.  Also, they defined the 
weighted past varextropy.  
\begin{definition}
Let $X$ be non-negative absolutely continuous  random variable. For $t>0$,
weighted residual varextropy and weighted past varextropy are defined, respectively, 
as follows
\begin{align}\label{wrve1} 
VJ^w_{\phi}(X_t)=\frac{1}{4}\left[\frac{1}{\overline F^3(t)}\int_{t}^{+\infty}\phi^2(x)f^3(x)dx-{\frac{1}{\overline F^4(t)}\left(\int_{t}^{+\infty}\phi(x) f^2(x)dx\right)}^2\right],
\end{align}
and 
\begin{align} 
VJ^w_{\phi}(X_{(t)})=\frac{1}{4}\left[\frac{1}{F^3(t)}\int_{0}^{t}\phi^2(x)f^3(x)dx-{\frac{1}{F^4(t)}\left(\int_{0}^{t}\phi(x) f^2(x)dx\right)}^2\right].
\end{align}
\end{definition} 
\begin{remark}
If hazard rate function $r(t)$ is increasing, then we can obtain 
\begin{align*}
VJ^w(X_t)+(J^w(X_t))^2&=\frac{1}{4\overline F^3(t)}\int_{t}^{+\infty}x^2{r}^3(x)\overline F^3(x)dx\\&\geq 
\frac{t^2 r^3(t)}{4}\int_{t}^{+\infty}\left(\frac{\overline F(x)}{\overline F(t)}\right)^3dx.
\end{align*}
\end{remark}
\begin{remark}
If we assume the reversed hazard rate function $\tilde{r}(t)$ is decreasing, then
we have 
\begin{align*}
VJ^w(X_{(t)})+{(J^w(X_{(t)}))^2}&=\frac{1}{4F^3(t)}\int_{0}^{t}x^2\tilde{r}^2(x)F^2(x)f(x)dx\nonumber\\&=\frac{1}{4F^3(t)}\left[\frac{x^2\tilde{r}^2(x)F^3(x)}{3}\Big|_{0}^{t}-\int_{0}^{t}\frac{F^3(x)}{3}\left(2x\tilde{r}^2(x)+2x^2\tilde{r}(x)\tilde{r}^{\prime}(x)\right)dx\right]
\\&\geq
\frac{\tilde{r}^2(t)t^2}{12}-\frac{1}{6}\int_{0}^tx\tilde{r}^2(x)\left(\frac{F(x)}{F(t)}\right)^3dx.
\end{align*}
\end{remark}
We now aim to analyze the behavior of the residual varextropy and the past varextropy. 
\begin{definition}
A random variable is said to be increasing (decreasing) in weighted residual varextropy if $VJ^w_{\phi}(X_t)$ is an increasing (decreasing) function of $t$.
\end{definition}
\begin{theorem}
Let $X$ is a non-negative random variable, then for $t>0$, $VJ^w_{\phi}(X_{t})$ is incresing  (decreasing) if and only if 
\begin{align}
VJ^w_{\phi}(X_{t})\geq (\leq) \frac{1}{12}\phi^2(t)(r(t))^2+\frac{1}{3}r(t)J^w_{\phi}(X_{t})\left[J^w_{\phi}(X_{t})+\phi(t)r(t)\right].
\end{align}
\end{theorem}
\begin{proof}
Differentiating \eqref{wrve1} with respect to $t$ leads to 
\begin{align}
\frac{d}{dt}VJ^{w}_{\phi}(X_{t})&=\frac{1}{4}\left\{-\phi^2(t)\left(\frac{f(t)}{\overline F(t)}\right)^3+\frac{3f(t)}{\overline F(t)}\int_{t}^{+\infty}\phi^2(x)\left(\frac{f(x)}{\overline F(t)}\right)^3dx\right.\nonumber
\\&+2\phi(t)\left(\frac{f(t)}{\overline F(t)}\right)^2\int_{t}^{+\infty}\phi(x)\left(\frac{f(x)}{\overline F(t)}\right)^2dx\nonumber
\left.-4\frac{f(t)}{\overline F(t)}\left(\int_{t}^{+\infty}\phi(x)\left(\frac{f(x)}{\overline F(t)}\right)^2dx\right)^2\right\}\nonumber
\\&=-\frac{1}{4}{\phi}^2(t)(r(t))^3+3r(t)VJ^w_{\phi}(X_{t})-r(t)J^w_{\phi}(X_{t})\left[J^w_{\phi}(X_{t})+\phi(t)r(t)\right].
\end{align}
Then $VJ^{w}_{\phi}(X_{t})$ is increasing (decreasing) if and only if 
\[-\frac{1}{4}{\phi}^2(t)(r(t))^3+3r(t)VJ^w_{\phi}(X_{t})-r(t)J^w_{\phi}(X_{t})\left[J^w_{\phi}(X_{t})+\phi(t)r(t)\right]\geq(\leq) 0,\]
but $r(t)\geq 0$, so the desired result is obtaind.
\end{proof}
\begin{remark}
Similarly, if $VJ^w_{\phi}(X_{(t)})$ is an increasing (decreasing) function of $t$ then we can obtain bounds for it as follows  
\begin{align}
VJ^w_{\phi}(X_{(t)})\leq (\geq) \frac{1}{12}\phi^2(t)(\tilde r(t))^2+\frac{1}{3}\tilde r{(t)}J^w_{\phi}(X_{(t)})\left[J^w_{\phi}(X_{(t)})+\phi(t)r(t)\right].
\end{align}
\end{remark}
Now, in the following theorem,  we obtain upper bound for the weighted residual varextropy $VJ^w_{\phi}(X_t)$, using mean residual life $m(t)$, given as
\[m(t)=E(X_t)=E(X-t|X>t).\]
\begin{theorem}
Let $X$ be a non-negative random variable with density function $f (x)$ and survival function $\overline F(x) = 1 - F(x)$. Then an upper bound for $VJ^w_{\phi}(X_t)$ is given by
\begin{align}
VJ^w_{\phi}(X_t)\leq \frac{1}{(\overline F(t))^2}E\left\{\Big[\phi(X_{t}+t)f^{\prime}(X_{t}+t)+\phi^{\prime}(X_t+t)f(X_{t}+t)\Big]^2\,\frac{m(t+X_{t})-m(t)+X_{(t)}}{r(t+X_{(t)})}\right\}.
\end{align}
\end{theorem}
\begin{proof}
In Proposition 1 of  \cite{Goodd}, we can show 
\begin{align}
\int_{x}^{+\infty}(y-m(t))\frac{f(t+y)}{\overline F(t)}dy=\frac{\overline F(t+x)}{\overline F(t)}\left\{m(t+x)-m(t)+x\right\},
\end{align}
are therefore by Lemma 1 in \cite{Goodd}, we have
\begin{align}
VJ&^w_{\phi}(X_{(t)})\leq \int_{0}^{+\infty}\left[\phi(y+t)\frac{f^{\prime}(y+t)}{\overline F(t)}+\phi^{\prime}(y+t)\frac{f(y+t)}{\overline F(t)}\right]^2\frac{\overline F(t+x)}{\overline F(t)}\left\{m(t+x)-m(t)+x\right\}dx\nonumber
\\&=\int_{0}^{+\infty}\left[\phi(y+t)\frac{f^{\prime}(y+t)}{\overline F(t)}+\phi^{\prime}(y+t)\frac{f(y+t)}{\overline F(t)}\right]^2\frac{1}{r(t+x)}\left\{m(t+x)-m(t)+x\right\}\frac{f(x+t)}{\overline F(t)}dx.
\end{align}
\end{proof}
In the following, we  provides a useful expression for the varextropy of the lifetime of coherent systems. Coherent systems offer a mathematical model for complex technical devices made of simple components. Notably, a structure consisting of $n$ components is known as a coherent system if it has no irrelevant components (a component is irrelevant if it does notmatter whether or not it is working)
and the system is monotone in every component.  

Let $X_1$, $\ldots$, $X_n$ be independent and identically distributed random variables with distribution function $F(t)$ and probability density function $f(t)$. They represent the lifetimes of the components of an $n$-component coherent system with lifetime $T$, and signature vector ${\bf s}=(s_1, s_2, \ldots, s_n)$, where $s_i=P(T=X_{i: n})$ for $i=1, \ldots, n$ is 
the probability that the $i$th component in the system is the last failed component and $\sum_{i=1}^ns_i=1$.
\cite{Saman} showed that the probability density function of $T$
is $f_T(t)=\sum_{i=1}^ns_if_{i: n}(t)$ where
\[f_{i: n}(t)=\frac{\Gamma(n+1)}{\Gamma(i)\Gamma(n-i+1)}[F(t)]^{i-1}[\overline F(t)]^{n-i}f(t), \, \, t\geq 0.\] Hence, 
the varextroy of the system with lifetime $T$ can be expressed as follows
\begin{align}
VJ(T)&=\frac{1}{4}\left[\int_{0}^{+\infty}(f_T(x))^3dx-\left(\int_{0}^{+\infty}(f_T(x))^2dx\right)^2\right]\nonumber\\&=\frac{1}{4}\left[\int_{0}^{+\infty}\left(\sum_{i=1}^ns_if_{i:n}(x)\right)^3dx-\left(\int_{0}^{+\infty}\left(\sum_{i=1}^ns_if_{i:n}(x)\right)^2dx\right)^2\right]\nonumber
\\&=\frac{1}{4}\left[\int_{0}^{1}\left(\sum_{i=1}^n s_ig_{i}(u)\right)^3f^2(F^{-1}(u))du-\left(\int_{0}^{1}\left(\sum_{i=1}^n s_ig_{i}(u)\right)^2f(F^{-1}(u))du\right)^2\right]\nonumber
\\&=\frac{1}{4}\left[\int_{0}^{1}g_V^3(u)f^2(F^{-1}(u))du-\left(\int_{0}^{1}g_V^2(u)f(F^{-1}(u))du\right)^2\right],
\end{align}
where $g_V(u)=\sum_{i=1}^n s_ig_{i}(u)$.
\begin{exmp}
Consider a coherent system of order 3 with lifetime components identically independent distributed, which have common exponential distribution with mean $\frac{1}{\lambda}$.  We consider  signature vector ${\bf s}=(\frac{1}{3}, \frac{2}{3}, 0)$ for the cohorent sysytem. It can be easily shown that $g_V(u)=(1-u)(1+3u)$, hence
\begin{align*}
J(T)=-\frac{\lambda}{2}\int_{0}^1{(1-u)}^3{(1+3u)}^2du=-0.35\lambda.
\end{align*} 
On the other hand, since 
$\frac{1}{4}\int_{0}^{+\infty}f^3_T(x)dx=\frac{\lambda^2}{4}\int_{0}^1(1-u)^5(1+3u)^3du=\frac{25}{168}\lambda^2$, thus
$VJ(T)=0.0263\lambda^2$.
\end{exmp}
When the structure of the engineering system is highly complex and contains a large number of components, calculating  $J(T)$ and consequently $VJ(T)$ becomes difficult or very time-consuming. In the following, we derive an upper bound on the 
lifetime of the coherent system, which can be useful for investigating the uncertainty behavior of its lifetime. 
\begin{theorem}\label{harth}
Let $T$ is lifetime of a cohorent system, consisting of $n$  independent and identically distributed components with  lifetimes  $X_1$, $\ldots$,  $X_n$  and a common distribution $F$ with probability density function $f$. Also, suppose that  signature of  the system is ${\bf s} = (s_1, ..., s_n)$. In this case, we have the following statement:  
\begin{align}
J(T)\leq \frac{-1}{8}\int_{0}^{1}\frac{1}{(F^{-1}(u))^2}\left(\sum_{i=1}^{n}s_iG_i(u)\right)^2dF^{-1}(u),
\end{align}
where $G_i(u)=\sum_{j=i}^n{n\choose j}u^j(1-u)^{n-j}$.
\end{theorem}
\begin{proof}
By using of Hardy's inequality  \eqref{T1.1} and equation \eqref{ex} for non-negative random variable, we have 
\begin{align*}
\int_{0}^{+\infty}f_T^2(x)dx&\geq \frac{1}{4}\int_{0}^{+\infty}{\left(\frac{1}{x}\int_{0}^{x}f_T(t)dt\right)}^2dx
= \frac{1}{4}\int_{0}^{+\infty}\frac{1}{x^2}\left(\int_{0}^x\sum_{i=1}^ns_if_{i: n}(t)dt\right)^2dx
\\&=\frac{1}{4}\int_{0}^{+\infty}\frac{1}{x^2}\left(\sum_{i=1}^ns_i\int_{0}^xf_{i: n}(t)dt\right)^2dx\\&=
\frac{1}{4}\int_{0}^{+\infty}\frac{1}{x^2}\left(\sum_{i=1}^ns_i\sum_{j=i}^n{n\choose j}F^j(x){(1-F(x))}^{n-j}\right)^2dx
\\&=\frac{1}{4}\int_0^1\frac{1}{{(F^{-1}(u))}^2}\left(\sum_{i=1}^ns_i\sum_{j=i}^n{n\choose j}u^j{(1-u)}^{n-j}\right)^2dF^{-1}(u),
\end{align*}
and thus, by utilizing of definition $G_i(u)$, the proof is completed. 
\end{proof}
\begin{theorem}
Under the conditions of Theorem \ref{harth}, we have
\begin{align}
VJ(T)&\leq \frac{2}{27}\int_{0}^{1}\frac{1}{(F^{-1}(u))^3}\Big(\sum_{i=1}^{n}s_iG_i(u)\Big)^3dF^{-1}(u)\nonumber\\&-{\left(\frac{1}{8}\int_{0}^{1}\frac{1}{(F^{-1}(u))^2}\Big(\sum_{i=1}^{n}s_iG_i(u)\Big)^2dF^{-1}(u)\right)}^2.
\end{align}
\end{theorem}
\begin{proof}
 The proof is similar to the previous theorem  and using Hardy's inequality.
\end{proof}
\begin{exmp}
Let $T$ denote the lifetime of a coherent system with the signature $\bold{s}=(0, \frac{1}{6}, \frac{7}{12}, \frac{1}{4})$, consisting of $n=4$ identically and independently distributed components. If the lifetimes of the components follow a common distribution $f(x)=\frac{\beta}{2} {(\frac{x}{2})}^{\beta-1},\, 0<x<2,$ and considering  that, $F^{-1}(x)=2x^{\frac{1}{\beta}}$ and $f(F^{-1}(x))=\frac{\beta}{2}x^{1-\frac{1}{\beta}}$, the upper bounds for extropy and varextropy are, respectively, equal to \[UJ=\frac{1}{16}\frac{-2168\beta^4+1316\beta^3-291\beta^2+28\beta-1}{6720\beta^5-5944\beta^4+2070\beta^3-355\beta^2+30\beta-1}\] and 
\[\frac{1}{108}\frac{5715\beta^4-4548\beta^3+1318\beta^2-168\beta+8}{12474\beta^5-14841\beta^4+6939\beta^3-1594\beta^2+180\beta-8}-(UJ)^2.\]
For example, for $\beta=2$, the upper bounds are equal to $-0.01169$ and $0.002494$. 
\end{exmp}
\begin{exmp}
Let $T$ represent  the lifetime of a coherent system with the signature given in the previous example. If the lifetimes of the components has a common exponential distribution with a mean of $\frac{1}{2}$, then $VJ(T^X) = 0.03659$. Additionally, if the lifetimes have a log-logistic distribution with pdf 
$g_1(x)=\frac{8x}{4x^2+1}, \, x>0$ 
then $VJ(T^{Y1})=0.04252$. Since $f(F^{-1}(x))=2(1-x)$  and $g_1(G_1^{-1}(x))=4\sqrt{x{(1-x)}^3}$, it is evident, by drawing a figure,  that
$X\leq_{dips}Y_1$. On the other hand, it has been shown that $T^X\leq_{VJ}T^{Y_1}$. 
Now, if the lifetimes of components follow a type-II pareto distribution with
$g_2(x)=\frac{5}{3}{(1+\frac{x}{3})}^{-6}, x>0$, then $g_2(G_2^{-1}(x))=\frac{5}{3}{(1-u)}^{\frac{6}{5}}$. It can be similarly demonstrated that $X\leq_{dips}Y_2$, whereas $VJ(T^{Y_2})=0.02215$, and thus $T^X\geq_{VJ}T^{Y_2}$. 
\end{exmp}
\begin{remark}
\cite{Vaselabadi} stated that if $X$ and $Y$ are two random variables such that $X\leq_{disp}Y$, then $X\geq_{VJ}Y$.
However, the above example provides a counterexample that challenges the validity of this proposition. 
\end{remark}
\section{Stochastic comparisons }
In this section, we propose a new stochastic order derived from the weighted varextropy measure.
\begin{definition}
The random variable $X$ is said to be smaller than $Y$ in the weighted varextropy order, denoted $X\overset{VJ^w}{\leq}Y$, if $VJ^w(X)\leq VJ^w(Y)$.
\end{definition}
In the following example, we get some comparisons about the weighted varextropy.
\begin{exmp}
\begin{enumerate}
\item[(a)] If $X$ has given Laplace distribution in Example \ref{eglap} with $\beta=1$ and $Y$ has standard exponential distribution then $VJ^W(x)=\frac{1}{216}$ and  $VJ^w(Y)=\frac{5}{1728}$ and thus $Y\overset{VJ^w}{\leq}X$.
\item[(b)] If $X$ and $Y$ follow Weibull distribution with $\lambda=1$ and $\alpha=2$ and $\lambda=1$ and $\alpha=1$, respectively. Then $VJ^w(X)=0.058$ and $VJ^w(Y)=0.0029$,  then $Y\overset{VJ^w}{\leq}X$.
\item[(c)] If $X\sim N(\mu_1, \sigma^2)$ and $Y\sim N(\mu_2, \sigma^2)$  and $\mu_1\leq \mu_2$ then using Example \ref{egnor} $X\overset{VJ^w}{\leq}Y$.  Also if $X\sim N(\mu, \sigma_1^2)$ and $Y\sim N(\mu, \sigma_2^2)$ and $\sigma_1\leq \sigma_2$ then $Y\overset{VJ^w}{\leq}X$.
\end{enumerate}
\end{exmp}
\begin{remark}
According to the above example and the results obtained from Example 6 in \cite{Vaselabadi}, it can be concluded that the varextropy order does not imply the weighted varextropy order.
\end{remark}
\begin{proposition}\label{p2.22}
The random variable $X$ has reciprocal distribution with probability distribution function 
\begin{align}\label{recip}
f(x)=\frac{1}{x[\ln(b)-\ln(a)]},\, a<x<b, \, \, a>0, 
\end{align}
if and only if $VJ^w(X)=0$.
\end{proposition}
\begin{proof}
The sufficiency part of the proof is straightforward. To prove of  necessity, we assume that $VJ^w(X)=0$, hence 
$xf(x)=E(Xf(X))=constant=c$ and therefore $f(x)=\frac{c}{x}$. Now, if we assume that $X$ has finite support $(a, b)$, then we conclude that $X$ has a reciprocal distribution with the probability density function given in \eqref{recip}. 
\end{proof}
\begin{corollary}
If $X$ has a reciprocal distribution, then for any continuous random variable $Y$,  we have $X\overset{VJ^w}{\leq}Y$, 
\end{corollary}
\begin{proposition}
Let $X_{k:n}$ denote the kth order statistic of the reciprocal distribution. Then 
\begin{itemize}
\item[a)] $X_{k:n}\overset{VJ^w}{\leq}X_{1:n}$ and $X_{k:n}\overset{VJ^w}{\leq}X_{n:n}$ for all $1\leq k\leq n$.
\item[b)] when $n$ is even, we have $X_{\frac{n}{2}: n}\overset{VJ^w}{\leq} X_{k:n}$ for all $1\leq k\leq n$.
\item[c)] when $n$ is odd, we have $X_{\frac{n+1}{2}: n}\overset{VJ^w}{\leq} X_{k:n}$ for all $1\leq k\leq n$.
\end{itemize}
\end{proposition}
\begin{proposition}
Let $X_1, X_2, \ldots, X_n$ be a random sample from a distribution with an absolutely continuous cumulative distribution  $F(x)$ and probability density function $f(x)$, then
\begin{align}
VJ^w(X_{r:n})&=\frac{B(3r-2, 3(n-r)+1)}{4B^3(r, n-r+1)} E[(F^{-1}(V1))^2f^2(F^{-1}(V1))]\nonumber\\&-
\frac{B^2(2r-1, 2(n-r)+1)}{4B^3(r, n-r+1)} E^2[F^{-1}(V2)f(F^{-1}(V2))].
\end{align}
where $V_1\sim Beta(3r-2, 3(n-r)+1)$ and $V_2\sim Beta(2r-1, 2(n-r)+1)$.
\end{proposition}
\begin{proposition}
Let $X_1, X_2, \ldots, X_n$ be a random sample from a symmetric distribution around zero with an absolutely continuous cumulative distribution  $F(x)$ and probability density function $f(x)$, then
\begin{align}
VJ^w(X_{r:n})=VJ^w(X_{n-r+1:n}).
\end{align}
\end{proposition}
\begin{proof}
We have 
\begin{align}
VJ^w(X_{r:n})=&\frac{1}{B(r, n-r+1)}\int_{-\infty}^{+\infty}x^2(F(x))^{r-1}(1-F(x))^{n-r}f(x)dx\\&=\frac{1}{B(r, n-r+1)}\int_{-\infty}^{+\infty}y^2(F(-y))^{r-1}(1-F(-y))^{n-r}f(-y)dy.
\end{align}
Now, since $F(-y)=1-F(y)$ and $f(-y)=f(y)$, hence
\begin{align}
VJ^w(X_{r:n})&=\frac{1}{B(n-r+1, r)}\int_{-\infty}^{+\infty}y^2(F(y))^{n-r}(1-F(y))^{r-1}f(y)dy=VJ^w(X_{n-r+1:n}),
\end{align}
and hence the desired result is obtained.
\end{proof}
The following examples are presented to demonstrate the weighted varextropy of the order statistics $X_{r: n}$ derived from selected probability distributions.
\begin{exmp}
If $X$ is uniformly distributed on $(a, b)$, then
\begin{align*}
VJ^w(X_{r:n})&=\frac{1}{4}\left[\frac{B(3r, 3(n-r)+1)}{B^3(r,n-r+1)}-\frac{B^2(2r, 2(n-r)+1)}{B^4(r,n-r+1)}\right]\\&+\frac{a^2}{4(b-a)^2}
\left[\frac{B(3r-2, 3(n-r)+1)}{B^3(r,n-r+1)}-\frac{B^2(2r-1, 2(n-r)+1)}{B^4(r,n-r+1)}\right]
\\&+\frac{a}{2(b-a)}\left[\frac{B(3r-1, 3(n-r)+1)}{B^3(r,n-r+1)}-\frac{B(2r, 2(n-r)+1)B(2r-1, 2(n-r)+1)}{B^4(r,n-r+1)}\right].
\end{align*}
\end{exmp}
\begin{exmp}
If $X$ has exponential distribution with mean $\theta$, then 
\end{exmp}
\begin{align*}
VJ^w(X_{r:n})&=\frac{B(3r-2, 3n-3r+3)}{B^3(r, n-r+1)}\left[\Psi(1,3n-3r)-\frac{1}{9}\left(\Psi(1,n+\frac{1}{3})+\Psi(1,n+\frac{2}{3})+\Psi(1,n)\right)\right.
\\&+(\Psi(3n-3r)-\ln 3)^2+\frac{1}{9}\left(\Psi^2(n+\frac{1}{3})+\Psi^2(n+\frac{2}{3})+\Psi^2(n)+2\Psi(n+\frac{1}{3})\Psi(n+\frac{2}{3})\right.\\&\left.+2\Psi(n+\frac{1}{3})\Psi(n)+2\Psi(n+\frac{2}{3})\Psi(n)\right)\\&+2\Big(\frac{1}{(3n-3r)(3n-3r+1)}+\frac{1}{(3n-3r)(3n-3r+2)}
+\frac{1}{(3n-3r+1)(3n-3r+2)}\Big)
\\&-2\left(\ln(3)-\Psi(3n-3r)\right)\left(\frac{1}{3n-3r}+\frac{1}{3n-3r+1}+\frac{1}{3n-3r+2}\right)\\&-\frac{2}{3}\left(\Psi(n+\frac{1}{3})+\Psi(n+\frac{2}{3})+\Psi(n)+\frac{1}{n}\right)\\&\left.\times\left(\Psi(3n-3r)-\ln(3)-\frac{1}{3n}+\frac{1}{3n-3r}+\frac{1}{3n-3r+1}+\frac{1}{3n-3r+2}\right)\right]\\&-\frac{B^2(2r-1, 2n-2r+2)}{B^4(r, n-r+1)}\left(\Psi(2n-2r)+\frac{1}{2n-2r}+\frac{1}{2n-2r+1}-\ln 2-\frac{1}{2}\Big(\Psi(n+\frac{1}{2})+\Psi(n)+\frac{1}{n}\Big)\right)^2.
\end{align*}
\section{Kernel estimation}
Let $X_1,\ldots, X_n$ be non-negative independent observations from density function $f$. The kernel density estimator for $f$ is defined by \cite{Parzen1962} as 
\begin{equation}\label{eq1}
f_n(x)=\frac{1}{nh_n}\sum_{i=1}^{n}K\big(\frac{x-X_i}{h_n}\big),
\end{equation}
where $K$ is a density kernel function and $h_n\rightarrow 0$ as $n\rightarrow \infty$ is a bandwidth sequence. Based on $f_n(x)$ \cite{Zhang2025} proposed the following plug-in estimator for $VJ^w(X)$ as bellow
\begin{equation}\label{eq2}
\hat{V}J_n^w(X)=\frac{1}{4}\left(\int_{0}^{+\infty}x^2f_n^3(x) dx-{\Big(\int_{0}^{+\infty}xf_n^2(x) dx\Big)}^2\right).
\end{equation}
\cite{Zhang2025} investigated some asymptotic properties for $\hat{V}J_n^w(X)$ in $\eqref{eq2}.$ $\hat{V}J_n^w(X)$ can be estimated by using numerical integration. When the dimensionality of random variable $X$ is high, numerical integration becomes unstable and computationally demanding. \cite{Joe} suggested using the second form of estimation that is resubstitution instead of numerical integration. Inspired by the estimator proposed by \cite{Joe}, we propose the following estimator for ${V}J^w(X)$:
\begin{align}\label{eq3}
\tilde{V}J^w(X) &=\frac{1}{4}\left(\int_0^{+\infty}x^2 {(f_n(x))}^2 dF_n(x)-{\Big(\int_0^{+\infty}x f_n(x) dF_n(x)\Big)}^2\right)\nonumber\\
&= \frac{1}{4n}\sum_{i=1}^{n}X_i^2{(\tilde{f}_{n}(X_i))}^2-\frac{1}{4}{\Big(\frac{1}{n}\sum_{i=1}^{n}X_i \tilde{f}_n(X_i)\Big)}^2,
\end{align}
where $F_n(x)=\frac{1}{n}\sum_{i=1}^{n}I(X_i\leq x)$ is the empirical distribution function and 
\begin{equation}\label{eq4}
\tilde{f}_n(X_i)=\frac{1}{(n-1)h_{n-1}}\sum_{j=1, j\neq i}^{n}K\Big(\frac{X_i-X_j}{h_{n-1}}\Big).
\end{equation} 
The weighted varextropy $VJ^w(X)$ can be rewritten as
\begin{equation}\label{eq5}
VJ^w(X)=\frac{1}{4}\left(\int_0^1 \frac{{(Q(u))}^2}{{(q(u))}^2} du-{\Big(\int_0^1 \frac{Q(u)}{q(u)} du\Big)}^2\right),
 \end{equation}
 where $Q(u)=\inf \{x ; F(x)\geq u\}, ~ 0\leq u \leq 1,$ is the quantile function corresponding to the distribution function $F(x)$ and $q(u)=Q'(u)=\frac{1}{f(Q(u))}$ is its quantile density function. Let $X_{(i)}$ be the $ith$ order statistic, a kernel estimator for $q(u)$ is given by \cite{Soni} as 
 \begin{equation}\label{eq6}
 \tilde{q}_n(u)=\frac{1}{h_n}\int_0^1 \frac{K(\frac{t-u}{h_n})}{f_n(Q_n(t))}dt=\frac{1}{nh_n}\sum_{i=1}^{n}
 \frac{K(\frac{S_i-u}{h_n})}{f_n(X_{(i)})},
 \end{equation} 
 where $S_i$ is the proportion of observations less than or equal to $X_{(i)}$ and $Q_n(u)=\inf\{ x; F_n(x)\geq u\}, 0\leq u \leq 1,$ is the empirical estimator of $Q(u).$ Weighted varextropy in the form of $\eqref{eq5}$ proposes the third estimator for weighted varextropy based on $\tilde{q}_n(u)$ and $Q_n(u)$ as
\begin{equation}\label{eq7}
\bar{V}J^w(X)=\frac{1}{4}\left(\int_0^1 \frac{{(Q_n(u))}^2}{{(\tilde{q}_n(u))}^2} du-{\Big(\int_0^1\frac{Q_n(u)}{\tilde{q}_n(u)}du\Big)}^2\right).
\end{equation}
We list the assumptions used in this section:
\begin{description}
\item[$(1)$]
$K(\cdot)$ has finite support.
\item[$(2)$]
$K$ is symmetric about zero.

\item[$(3)$]
There is a positive constant $M$ such that 
$ 
\left|K(x)-K(y)\right|\leq M \left|x-y\right|.
$
\end{description}
In the following, we investigate the almost sure (a.s.) consistency of the proposed estimators.
\begin{theorem}
If $f$ is bounded and $E(X^2)<\infty,$ then we have
\begin{equation*}
\lim_{n \rightarrow \infty}\tilde{V}J^w(X)=VJ^w(X) ~ a.s.
\end{equation*}
\end{theorem}
\begin{proof}
\begin{align}\label{eq8}
\tilde{V}J^w(X)-VJ^w(X) &\leq \frac{1}{4} \sup_x\left|f_n^2(x)-f^2(x)\right|\int_0^{+\infty} x^2 dF_n(x)\nonumber\\
&+\frac{1}{4} \left(\int_0^{+\infty} x^2 f^2(x) dF_n(x)- \int_{0}^{+\infty} x^2 f^3(x) dx\right)\nonumber\\
&+ \frac{1}{4}\left({\Big(\int_0^{+\infty}xf(x)dF(x)\Big)}^2-{\Big(\int_{0}^{+\infty}xf_n(x)dF_n(x)\Big)}^2\right)\nonumber\\
&:=I_1+I_2+I_3. 
\end{align}
The Kolmogorov law of large numbers implies that
\begin{align}\label{eq9}
\int_0^{+\infty}x^2 dF_n(x)=\frac{1}{n}\sum_{i=1}^{n}X_i^2\rightarrow E(X^2)~ a.s.,
\end{align}
\begin{align}\label{eq10}
\int_{0}^{+\infty}x^2 f^2(x)dF_n(x)=\frac{1}{n}\sum_{i=1}^{n}X_i^2 f^2(X_i)\rightarrow \int_0^{+\infty}x^2 f^3(x) dx<\infty~ a.s.
\end{align}
For $n\rightarrow \infty,$ $\sup_x\left|f_n(x)-f(x)\right|\rightarrow 0,$ (see \cite{Prakas}). This last result, $\eqref{eq9}$ and $\eqref{eq10}$ conclude that
\begin{align}\label{eq11}
\lim_{n \rightarrow \infty}I_1=0~a.s.,
\end{align}
and 
\begin{align}\label{eq12}
\lim_{n \rightarrow \infty}I_2=0~a.s.,
\end{align}
In a similar way we can see that 
\begin{equation*}
\lim_{n \rightarrow \infty}\int_{0}^{\infty}xf_n(x) dF_n(x)=\int_{0}^{\infty}x f(x) dF(x),
\end{equation*}
and hence
\begin{align}\label{eq13}
\lim_{n \rightarrow \infty}I_3=0~a.s.,
\end{align}
$\eqref{eq8}, \eqref{eq11}, \eqref{eq12}$ and $\eqref{eq13}$ complete the proof.
\end{proof}
\begin{theorem}
Suppose that $f$ is bounded and $\tau=\sup\{x ; F(x)<1\}<\infty.$

Then as $n \rightarrow \infty$
\begin{align*}
\bar{V}J^w(X)\rightarrow VJ^w(X).
\end{align*}
\end{theorem}
\begin{proof}
\begin{align}\label{I}
\bar{V}J^w(X)-VJ^w(X)&= \frac{1}{4}\int_{0}^1 \frac{1}{{(\tilde{q}_n(u))}^2}\left({(Q_n(u))}^2-{(Q(u))}^2\right) du\nonumber\\
&+\frac{1}{4}\int_0^1 {(Q(u))}^2\left[\frac{1}{{(\tilde{q}_n(u))}^2}-\frac{1}{{(q(u))}^2}\right]du\nonumber\\
&+\frac{1}{4}\left({\Big(\int_0^1 \frac{Q(u)}{q(u)}du\Big)}^2-{\Big(\int_0^1 \frac{Q_n(u)}{\tilde{q}_n(u)}du\Big)}^2\right)\nonumber\\
&:=\frac{1}{4}I_{n1}+\frac{1}{4}I_{n2}+\frac{1}{4}I_{n3}.
\end{align}
On the other hand
\begin{align*}
\frac{1}{{(\tilde{q}_n (u))}^2}-\frac{1}{{(q(u))}^2}=\left(q(u)-\tilde{q}_n(u)\right)\frac{q(u)+\tilde{q}_n(u)}{{(\tilde{q}_n(u))}^2{(q(u))}^2}.
\end{align*}
Using the proof of Theorem $3.4$ of \cite{Subh}, $\sup_u\left|\tilde{q}_n(u)-q(u)\right|\rightarrow 0$ as $n \rightarrow \infty,$ which implies that
\begin{align}\label{II}
\lim_{n \rightarrow \infty}\sup_u \left| \frac{1}{{(\tilde{q}_n (u))}^2}-\frac{1}{{(q(u))}^2}\right|
\leq 2 \lim_{n \rightarrow \infty}\sup_u \left|q(u)-\tilde{q}_n(u)\right|{(\sup_x f(x))}^3=0.
\end{align}
To deal with the term $I_{n2}$  we can observe that
\begin{align}\label{III}
\lim_{n \rightarrow \infty}\left|I_{n2}\right|\leq \lim_{n \rightarrow \infty}\sup_u \left| \frac{1}{{(\tilde{q}_n (u))}^2}-\frac{1}{{(q(u))}^2}\right|\int_0^1 x^2 f(x) dx=0.
\end{align}
Next
\begin{align*}
\sup_{0<u<1}\left|Q_n^2(u)-Q^2(u)\right|&\leq
\sup_{0<u<1}\left|Q_n(u)-Q(u)\right|\{2|Q_n(u)|+\sup_{0<u<1}|Q(u)-Q_n(u)|\}\nonumber\\
&\leq \sup_{0<u<1}\left|Q_n(u)-Q(u)\right|\{2\tau+\sup_{0<u<1}\left|Q(u)-Q_n(u)\right|\}.
\end{align*}
For $n \rightarrow \infty, \sup_{0<u<1}\left|Q_n(u)-Q(u)\right|\rightarrow 0$ which implies that
\begin{align}\label{IV}
\lim_{n \rightarrow \infty}\sup_{0<u<1} \left| Q_n^2(u)-Q^2(u)\right|=0.
\end{align}
Now
\begin{align*}
|I_{n1}|\leq\sup_{0<u<1}\left|Q_n^2(u)-Q^2(u)\right|\left(\left|\int_0^1 \big(\frac{1}{{(\tilde{q}_n(u))}^2}-\frac{1}{{(q(u))}^2}\big)du\right|+{|sup_x f(x)|}^2\right).
\end{align*}
This last result, $\eqref{II}$ and $\eqref{IV}$ imply that
\begin{align}\label{V}
\lim_{n \rightarrow \infty}I_{n1}=0.
\end{align}
In a similar way we can show that
\begin{align}\label{VI}
\lim_{n \rightarrow \infty}\int_0^1 \frac{Q_n(u)}{\tilde{q}_n(u)}du=\int_{0}^1 \frac{Q(u)}{q(u)}du,
\end{align}
which implies that
\begin{align}\label{VII}
\lim_{n \rightarrow \infty}I_{n3}=0.
\end{align}
$\eqref{I}, \eqref{III}, \eqref{V}$ and $\eqref{VII}$ complete the proof.
\end{proof}
\section{Simulation}
In this section, we present  the Monte Carlo studies on the biases and mean squared errors (MSEs) to demonstrate the efficacy of the proposed estimators. We generate random samples of different sizes from the Gamma(2,1) and Beta(2,1) distributions. 
For each sample size $n=10, 20, 30, 50, 100$, a total of 10,000 samples are drawn. 
The Epanechnikov kernel $K(x)=\frac{3}{4}(1-x^2), ~|x|<1$ is used as the kernel function in all the cases. 
We compute the bias and MSE for each estimators $\hat{V}J^w(X), \tilde {V}J^w(X)$ and $\bar{V}J^w(X)$, using the bandwidth $h_n=1.06 s n^{-\frac{1}{5}},$ where $s$ is the sample standard deviation. 
Tables $\ref{ta1}$ and $\ref{ta2}$ show the simulated biases and MSEs of the proposed estimators for the gamma and beta distributed samples, respectively. 
From Tables $\ref{ta1} $ and $\ref{ta2}$, we observe that as the sample sizes increases, both the bias and the MSE decrease. In the case of the gamma distribution, Table $\ref{ta1},$ the estimator $\tilde{V}J^w(X)$ performs best in terms of bias and MSE. Under the beta distribution, Table $\ref{ta2},$ the biases and MSEs of $\bar{V}J^w(X)$ are consistently smaller than those of the other proposed estimators, expect for sample sizes $n=10$ and $20.$ 
In the tables, we use the notations $\hat{B}, \tilde{B}$ and $\bar{B}$ to denote the biases of
 $\hat{V}J^w(X), \tilde {V}J^w(X)$ and $\bar{V}J^w(X)$, respectively. Also we employ $\hat{M}, \tilde{M}$ and $\bar{M}$ for MSE of $\hat{V}J^w(X), \tilde {V}J^w(X)$ and $\bar{V}J^w(X)$, respectively.
 
\begin{table}[h!]
\centering
\caption{MSE and bias of weighted varextropy estimators for the gamma(2,1).}\label{ta1}
\begin{tabular}{ccccccc}
\hline 
$n$ & $\hat{B}$ & $\tilde{B}$ & $\bar{B}$ &$\hat{M}$ & $\tilde{M}$ & $\bar{M}$\\ 
\hline 
 10 &0.004616  & 0.001667 & 0.082726 &0.000054& 0.000042&0.010144 \\ 
 20& 0.003340&0.001070&0.060086&0.000022&0.000014&0.004482\\
30&0.002845&0.000916&0.047863&0.000014&0.000009&0.002681 \\ 
 50 &0.002485&0.000767&0.034503&0.000010&0.000005&0.001327\\
100 &0.002075&0.000659&0.020492&0.000006&0.000003&0.000463\\
\hline
\end{tabular} 
\end{table}

\begin{table}[h!]
\centering
\caption{MSE and bias of weighted varextropy estimators for beta(2,1).}\label{ta2}
\begin{tabular}{ccccccc}
\hline 
$n$ & $\hat{B}$ & $\tilde{B}$ & $\bar{B}$ &$\hat{M}$ & $\tilde{M}$ & $\bar{M}$\\ 
\hline 
 10 &-0.031170  &-0.049123  & -0.041327 &0.002345&0.003305 &0.002517 \\ 
 20& -0.032032&-0.047520&-0.033247&0.001591&0.002632&0.001663 \\
30&-0.031514&-0.045645&-0.028686&0.001405&0.002367&0.001274 \\ 
 50& -0.030104&-0.043037&-0.024984 & 0.001197&0.002070&0.000970\\
100&-0.027820&-0.039374&-0.021544&0.000960&0.001699&0.000697 \\
\hline
\end{tabular} 
\end{table}
\section{Application}
In this section, a few goodness-of-fit tests for reciprocal distribution based on $\hat{V}J^w(X), \tilde{V}J^w(X)$ and $\bar{V}J^w(X)$ are suggested. Monte Carlo simulation is used to obtain the percentage points and power values of the tests. Consider the class of continuous distribution functions $F$ with density function $f$ defined on interval $[a,b]$ where $0<a<b.$ Proposition \ref{p2.22} gives us the idea to use the weighted varextropy estimators as the test statistics of goodness-of-fit tests of reciprocal distribution with distribution function $F(x)=\frac{\ln(\frac{x}{a})}{\ln(\frac{b}{a})},~a<x<b.$ Let $X_1,\ldots, X_n$ be a random sample from a continuous distribution function $F(x)$ on $[a,b].$ The null hypothesis is $H_0: F(x)=\frac{\ln(\frac{x}{a})}{\ln(\frac{b}{a})},~a<x<b,$ and the alternative hypothesis ($H_1$) is the opposite of $H_0.$ Given any significance level $\alpha,$ our hypothesis-testing procedure can be defined by the critical region:
\begin{align*}
G_n=VJ_n^w(X)\geq C_{1-\alpha},
\end{align*}
 where $VJ_n^w(X)$ is one of the estimators and $C_{1-\alpha}$ is the critical value for the test with level $\alpha.$ Since $VJ_n^w(X)$ converges to $VJ^w(X)$ in probability, under $H_0$, $G_n$ converges to  $0$ in probability, and under $H_1,$ $G_n$ converges to a number larger than zero in probability. We employ  following notations to indicate the test statistics for testing the reciprocal distribution
 \begin{align*}
 {\hat{G}}_n&=\hat{V}J^w(X),\nonumber\\
 {\tilde{G}}_n&=\tilde{V}J^w(X),\nonumber\\
  {\bar{G}}_n&=\bar{V}J^w(X). 
 \end{align*}
We compare the powers of our proposed test statistics with Kolmogorov-Smirnov statistic defined by \cite{Kol} and \cite{Smir} as
 \begin{align*}
 KS=\max\left(\max_{1\leq i \leq n}\left\{\frac{i}{n}-X_{(i)}\right\}, \max_{1 \leq i \leq n}\left\{ X_{(i)}-\frac{i-1}{n}\right\} \right),
  \end{align*}
where $X_{(1)}, \ldots, X_{(n)}$ are order statistics. We conduct simulation studies in two following scenarios:\\
{\bf Scenario $1.$} We suppose that $a=\frac{1}{4}$ and $b=1$ and compute the powers of the tests under the following alternative distribution:
\begin{align*}
A_k: F(x)=1-\frac{{(1-x)}^k}{{(1-a)}^k},~ a<x<1,~(for~ k=1.5,2).
 \end{align*}
{\bf Scenario $2.$} We suppose that $a=\frac{1}{4}$ and $b=10$ and compute the powers of the tests under the truncated lognormal distribution: 
\begin{align*} 
TL: F(x)=\frac{\Phi(\ln(x))-\Phi(\ln(\frac{1}{4}))}{\Phi(\ln(10))-\Phi(\ln(\frac{1}{4}))}. 
 \end{align*}
 The critical values are estimated based on $100000$ repetitions and shown in Tables $\ref{ta3}$ and $\ref{ta4}.$ The powers of proposed test statistics and KS statistic are obtained in Tables $\ref{ta5}$ and $\ref{ta6}.$ These powers are estimated based on $100000$ repetitions for $n=10,20,30,40,50$ and $\alpha=0.05.$ The statistic achieving the maximal power is indicated by the bold type in Tables $\ref{ta5}$
 and $\ref{ta6}.$ According to Tables $\ref{ta5}$ and $\ref{ta6},$ the performance of tests depends on alternative distributions. The test based on $\tilde{G}_n$ is the best for $A_k$ distribution, (Table $\ref{ta5}$). Also according to Table $\ref{ta6},$ the test based on $\hat{G}_n$  is the best for the truncated lognormal distribution.
\begin{table}[h!]
\centering
\caption{\small\small\small Percentage points of the proposed test statistics at the level $\alpha=0.05$ for $a=\frac{1}{4}$ and $b=1$. }\label{ta3}
\begin{tabular}{cccc}
\hline 
$n$ & $\hat{G}$ & $\tilde{G}$ & $\bar{G}$ \\ 
\hline 
 10 &0.046630  & 0.036781 & 0.035007  \\ 
20 & 0.034493 & 0.024466 & 0.024606  \\ 
30 & 0.029452 & 0.019787 & 0.021169  \\ 
40 &0.026457  & 0.017086 &  0.016751 \\ 
50 & 0.024594 &  0.015266& 0.015777  \\ 
75 &0.021494  & 0.012537 & 0.012400  \\ 
100 &0.019634  &0.010888  & 0.010811  \\ 
\hline
\end{tabular} 
\end{table} 
\begin{table}[h!]
\centering
\caption{\small\small\small Percentage points of the proposed test statistics at the level $\alpha=0.05$ for $a=\frac{1}{4}$ and $b=10$. }\label{ta4}
\begin{tabular}{cccc}
\hline 
$n$ & $\hat{G}$ & $\tilde{G}$ & $\bar{G}$ \\ 
\hline 
 10 &0.012298  & 0.009770 &  0.555766 \\ 
20 & 0.008748& 0.006930 & 0.387190 \\ 
30 &0.007379  &0.005752  & 0.312608  \\ 
40 & 0.006644 &0.005101  & 0.268846  \\ 
50 & 0.006147 &  0.004698& 0.240833  \\ 
75 & 0.005427 &0.004021  &  0.201696 \\ 
100 &0.005005  & 0.003643 & 0.170413 \\ 
\hline
\end{tabular} 
\end{table} 
\begin{table}[h!]
\centering
\caption{\small\small\small Power comparisons of the tests at the level $\alpha=0.05$ for $a=\frac{1}{4}$ and $b=1$. }\label{ta5}
\begin{tabular}{cccccc}
\hline 
$n$ &Alternative &$\hat{G}$ & $\tilde{G}$ & $\bar{G}$& KS \\ 
\hline 
 10 &$A_{1.5}$ & 0.090239&$\bold{ 0.096029}$& 0.086913& 0.065409 \\ 
10 & $A_2$& 0.078489&$\bold{ 0.088299}$& 0.036963& 0.067839  \\ 
20&$A_{1.5}$& 0.114898&$\bold{0.134688}$& 0.094905& 0.101729 \\ 
20 &$A_2$& 0.097299& $\bold{0.114868}$& 0.044955& 0.092649  \\ 
30 &$A_{1.5}$& 0.137328&$\bold{0.162648}$&0.099900& 0.132968 \\ 
30 & $A_2$ & 0.120338& $\bold{0.137818}$& 0.052947&0.094499 \\
40 & $A_{1.5}$ & 0.157698& $\bold{0.192638}$& 0.122877& 0.141478  \\
40 & $A_2$& 0.146538&$\bold{ 0.162988}$& 0.063936&0.083259  \\
50& $A_{1.5}$& 0.174088& $\bold{0.220397}$& 0.127872& 0.149608 \\
50& $A_2$ & 0.165368& $\bold{0.187738}$& 0.071928&0.076869  \\
\end{tabular} 
\end{table} 
\begin{table}[h!]
\centering
\caption{\small\small\small Power comparisons of the tests at the level $\alpha=0.05$ for $a=\frac{1}{4}$ and $b=10$. }\label{ta6}
\begin{tabular}{cccccc}
\hline 
$n$ & Alternative &$\hat{G}$ & $\tilde{G}$ & $\bar{G}$&KS \\ 
\hline 
 10 & TL &$\bold{0.089699 }$ & 0.081729& 0.002997& 0.005929  \\ 
20 & TL&$\bold{0.133068}$  &0.090909& 0.003036& 0.006139 \\ 
30 &TL &$\bold{0.240757 }$ &0.130278& 0.003721& 0.006069  \\ 
40 & TL &$\bold{0.378546}$ &0.183418& 0.003912& 0.006039   \\ 
50 &TL  &$\bold{0.520804}$  &0.243817& 0.004265& 0.005889   \\ 
\end{tabular} 
\end{table}
  \section{Real data}
\cite{Kayid} used $42$ dataset of COVID-19 infections gathered from various official sources as March 26, 2020. \cite{Kayid} showed that the data follows an exponential distribution with an estimated parameter $\hat{\lambda}=0.32.$ Table $\ref{ta7}$ shows the values of the proposed estimators based on this data. From Table $\ref{ta7}$ we can see the closeness estimators to the theoretical value $VJ^w(X).$ 
 \begin{table}[h!]
\centering
\caption{\small\small\small Theoretical value and the proposed estimators. }\label{ta7}
\begin{tabular}{cccc}
\hline 
$VJ^w(X)$ & $\hat{V}J^w(X)$ &$\tilde{V}J^w(X)$ &$\bar{V}J^w(X)$ \\ 
\hline 
 0.002893518&0.005168755&0.003840097&0.01658735 \\ 
\end{tabular} 
\end{table} 
\section{Conclusion}
In this article, we demonstrated through several examples
that weighted varextropy can be considered as an appropriate measure of variability when the extropy and varextropy are equal for a set of probability distributions. We further showed that, for some distributions, unlike the varextropy, the weighted varextropy does not depend on the parameter, which indicates the flexibility of this measure. 
We also derived bounds for this measure, as well as for the weighted residual and past  varextropy. In addition, an 
explicit expression for the system's lifetime varextropy was obtained under the assumption that the components of the system are independent and identically distributed. Moreover, we introduced a weighted varextropy order and showed that if $X$ follows a  reciprocal distribution, then for all random variable $Y$, the inequality $X\overset{VJ^w}\leq Y$ holds. 
We also showed that, contrary to Proposition 14 in \cite{Vaselabadi}, the condition $X\overset{disp}{\leq}Y$ does not necessarily imply that  $X\overset{VJ}{\geq}Y$. 
Finally, two estimators for the weighted varextropy were proposed, and their consistency was rigorously established. Several tests for the reciprocal distribution were developed using the proposed estimators, and their powers were compared with that of the Kolmogorov–Smirnov (KS) test. An application to real data was also presented.
%
%

\end{document}